\documentclass{amsart}
\newtheorem{thm}{Theorem}[section]
\newtheorem{cor}[thm]{Corollary}
\newtheorem{lem}[thm]{Lemma}

\theoremstyle{definition}

\newtheorem{exam}[thm]{Example}
\theoremstyle{remark}

\numberwithin{equation}{section}
\begin{document}

\title{Duality Between Uniform Spaces and Boolean Algebras}
\author{Joseph Van Name}
\begin{abstract}
In this note we shall generalize the Stone duality between compact totally
disconnected spaces and Boolean algebras to a duality between all
complete non-Archimedean uniform spaces and Boolean algebras.
\end{abstract}
\maketitle
\section{Boolean algebras}

In this note, if $f:X\rightarrow Y$ is a function and $A\subseteq X,B\subseteq Y$,
then we shall let $f''(A)$ be the image of $A$ and $f_{-1}(B)$ denote
the inverse image of $B$.

Let $P$ be a poset. Then $x,y\in P$ are said to be incompatible if
there does not exist an $r\in P$ with $r\leq x,r\leq y$. A subset
$A\subseteq P$ is said to be cellular if every pair of elements in
$A$ is incompatible.

If $(A,\wedge,0)$ is a semilattice, then we shall say $x,y\in A\setminus\{0\}$
are incompatible if $x\wedge y=0$, and we shall say $A'\subseteq A\setminus\{0\}$
is cellular if $x\wedge y=0$ for $x,y\in A',x\neq y$.

\begin{thm}
Let $P$ be a poset. Then every cellular family is contained in a
maximal cellular family (ordered under $\subseteq$).
\end{thm}
\begin{proof}
This is a simple application of Zorn's lemma. If $(R_{b})_{b\in B}$
is a chain of cellular families, then $\cup_{b\in B}R_{b}$ is cellular.
\end{proof}
Let $P$ be a poset(semilattice), the write $c(P)$ for the collection of cellular families
on $P$. If $A,B\in c(P)$, then write $A\preceq B$ if for each $a\in A$, there
is a $b\in B$ with $a\leq b$. Since $B$ is cellular, the $b\in B$ with
$a\leq b$ is unique, so let's write $\phi_{A,B}:A\rightarrow B$
for the unique function with $a\leq	\phi_{A,B}(a)$.

\begin{thm}
$c(P)$ is a poset under the ordering $\preceq$, and $c(P)$ is an inverse system
with the mappings $\phi_{A,B}$.
\end{thm}
\begin{proof}
Clearly $A\preceq A$, and for each $a\in A:\phi_{A,A}(a)=a$ since $a\leq a$.
If $A\preceq B,B\preceq A$, then for each $a\in A$,
we have $a\leq\phi_{A,B}(a)\leq\phi_{B,A}\phi_{A,B}(a)$, so since $A$ is
cellular, we have $A$ be an antichain, so $a=\phi_{B,A}\phi_{A,B}(a)$, so
$a=\phi_{A,B}(a)$, so $A\subseteq B$. Similarly, we have $B\subseteq A$,
so $A=B$. Now if $A\preceq B,B\preceq C$, then $a\leq\phi_{A,B}(a)\leq\phi_{B,C}\phi_{A,B}(a)$,
so $A\preceq C$, and $\phi_{A,C}=\phi_{B,C}\phi_{A,B}$.
\end{proof}

Let $B$ be a Boolean algebra. Then a partition $P$ of $B$ is a subset
of $B\setminus\{0\}$ where $\vee P=1$ and where $x\wedge y=0$ for $x\neq y$.
The partitions of a Boolean algebra are precisely the maximal elements of
$c(B)$ with the inclusion ordering $\subseteq$.
We shall write $\mathbb{P}(B)$ for the collection of all partitions of
a Boolean algebra $B$.

\begin{lem}
For Boolean algebras $B$ the collection of partitions of $B$ form a lower
semilattice where if $P,Q\in\mathbb{P}(B)$, then $P\wedge Q=\{p\wedge q|p\in P,q\in Q,p\wedge q\neq 0\}$.
\end{lem}

\begin{thm}
Let $p,q\in\mathbb{P}(B)$ and $p\preceq q$. Then if $p=\{a_{i}|i\in I\},q
=\{b_{j}|j\in J\}$, then for each $j\in J$ we have $b_{j}=\vee\{a_{i}|a_{i}\leq b_{j}\}$
\end{thm}
\begin{proof}
Let's assume that for some $j$ we do not have $b_{j}=\vee\{a_{i}|a_{i}\leq b_{j}\}$
, then there is an $r\in B$ with $r\geq a_{i}$ for $i\in I$ and $r<b_{j}$. We therefore have
$b_{j}\wedge r'\neq 0$. If $a_{i}\leq b_{j}$, then $a_{i}\leq r$, so 
$(b_{j}\wedge r')\wedge a_{i}\leq(b_{j}\wedge r')\wedge r=0$, and if
$a_{i}\not\leq b_{j}$, then $a_{i}\leq b_{j'}$ for some $j'\neq j$, so
$(b_{j}\wedge r')\wedge a_{i}\leq(b_{j}\wedge r')\wedge b_{j'}=0$, so
$(b_{j}\wedge r')\wedge a_{i}=0$ for each $i\in I$, so $\{a_{i}|i\in I\}$
is not a partition of $B$. This is a contradiction.
\end{proof}

A partition $p$ on a Boolean algebra $B$ is said to be subcomplete if whenever
$r\subseteq p$, $\vee r$ exists.

\begin{thm}
If a partition $p$ of a Boolean algebra $B$ is subcomplete and $p\preceq q$, then
$q$ is subcomplete as well.
\end{thm}
\begin{proof}
Since $p\preceq q$ for each $a\in q$ there is a $P_{a}\subseteq p$ with
$a=\vee P_{a}$. If $Q\subseteq q$, then
$\vee(\cup_{a\in Q}P_{a})=\vee_{a\in Q}\vee P_{a}=\vee Q$, so $q$ is subcomplete as well.
\end{proof}

\begin{thm}
Let $A$ be a subcomplete partition of a Boolean algebra $B$. Then the map
$\phi:P(A)\rightarrow B$ given by $\phi(R)=\vee R$ is an injective
Boolean algebra homomorphism with $\phi(\cup_{i\in I}R_{i})=\vee_{i\in I}\phi(R_{i})$
i.e. $\phi$ preserves all least upper bounds.
\end{thm}
\begin{proof}
Let $R_{i}\subseteq A$ for $i\in I$. Then
$\phi(\cup_{i\in I}R_{i})=\vee\cup_{i\in I}R_{i}=\vee_{i\in I}\vee R_{i}=\vee_{i\in I}\phi(R_{i})$.
Furthermore, $\phi(R)\vee\phi(R^{c})=\phi(R\cup R^{c})=\phi(A)=1$ and
$\phi(R)\wedge\phi(R^{c})=(\vee R)\wedge(\vee R^{c})=\vee_{a\in R,b\in R^{c}}(a\wedge b)
=0$, so $\phi(R^{c})=\phi(R)^{c}$. Therefore $\phi$ is a Boolean algebra
homomorphism preserving all least upper bounds, and $\phi$ is injective since
$Ker(\phi)$ is trivial.
\end{proof}

A Boolean partition algebra is a pair $(B,F)$ where $B$ is a Boolean algebra
and $F$ is a (possibly improper) filter on $\mathbb{P}(B)$ where $\{b,b'\}\in F$
for each $b\in B\setminus\{0,1\}$.

\begin{lem}
Let $F\subseteq\mathbb{P}(B)$ be a filter. Then the following are equivalent.

1. $(B,F)$ is a Boolean partition algebra.

2. For each $b\in B\setminus\{0\}$ there is a $P\in F$ with $b\in P$.

3. $F$ contains all partitions of $B$ into finitely many sets.
\end{lem}
\begin{proof}
$1\rightarrow 2$ Let $b\in B\setminus\{0\}$. If $b=1$, then $b\in\{b\}\in F$.
If $b\neq 1$, then $\{b,b'\}\in F$.

$2\rightarrow 3$ Let $\{b_{1},...,b_{n}\}$ be a partition of $B$. If $n=1$, then
$\{b_{1}\}\in F$. If $n>1$, then 
$\{b_{1},...,b_{n}\}\succeq \{b_{1},b_{1}'\}\wedge...\wedge\{b_{n},b_{n}'\}\in F$.

$3\rightarrow 1$ This is trivial.
\end{proof}
\begin{exam}
If $B$ is a Boolean algebra, and $\lambda$ is an infinite cardinal,
then define $\mathbb{P}_{\lambda}(B)=\{P\in\mathbb{P}(B):|P|<\lambda\}$.
Then $(B,\mathbb{P}_{\lambda}(B))$ is Boolean partition algebra.
\end{exam}

\begin{thm}
If $B$ is a Boolean algebra, and $F\subseteq\mathbb{P}(B)$ is a filter,
$\{0\}\cup\cup F$ is a subalgebra of $B$ and $(\{0\}\cup\cup F,F)$ is a
Boolean partition algebra.
\end{thm}
\begin{proof}
Let $a,b\in\{0\}\cup\cup F$. If $a\wedge b=0$, then $a\wedge b\in\{0\}\cup\cup F$.
If $a\wedge b\neq 0$, then $a\neq 0,b\neq 0$, so there are $p,q\in F$
with $a\in p,b\in q$, so $a\wedge b\in p\wedge q$, so $a\wedge b\in\{0\}\cup\cup F$.

Clearly $0\in\{0\}\cup\cup F$ and $1\in\{1\}$, so $1\in\{0\}\cup\cup F$.

If $a\in\{0\}\cup\cup F,a\neq 0,a\neq 1$, then $a\in p$ for some $p\in F$,
so $a'\in\cup F$. Therefore $\{0\}\cup\cup F$ is a subalgebra of $B$.

Clearly $F$ is closed under $\wedge$.
Now assume $p,q\in\mathbb{P}(\{0\}\cup\cup F),p\in F$ and assume $p\preceq q$.
Then we claim that $q$ is a partition of $B$. Clearly $q$ is a cellular family.
If $x\in B$ and $x\geq b$ for each $b\in q$, then for each $a\in p$ we have
a $b\in q$ with $a\leq b\leq x$, so $x=1$. We therefore have $q$ be a partition of $B$,
so $q\in F$. We therefore have $F$ be a filter on $\mathbb{P}(\{0\}\cup\cup F,F)$.
If $a\in(\{0\}\cup\cup F)\setminus\{0\}$, then $a\in p\in F$ for some $p\in F$,
so $(\{0\}\cup\cup F,F)$ is a Boolean partition algebra.
\end{proof}
If $B$ is a Boolean algebra and $F$ is a filter on $\mathbb{P}(B)$, then write
$\mathfrak{B}^{*}(B,F)$ for the Boolean partition algebra $(\{0\}\cup\cup F,F)$.
If $B=P(X)$ for some set $X$, then we shall write $\mathfrak{B}^{*}(X,F)$ for
$\mathfrak{B}^{*}(P(X),F)$.

Given a Boolean partition algebra $(B,F)$, and an ultrafilter $\mathcal{U}\subseteq B$,
then we shall call $\mathcal{U}$ an $F$-ultrafilter if for each $P\in F$ there is an
$a\in P\cap\mathcal{U}$ i.e. $|P\cap\mathcal{U}|=1$ for each $P\in F$.
We shall write $S_{F}(B)$ or $S^{*}(B,F)$ for the collection of all $F$-ultrafilters
on $B$, and we shall write $S^{*}(B)$ for the collection of all ultrafilters on $B$.

\begin{lem}
Let $(B,F)$ be a Boolean partition algebra, and let $x=(x_{p})_{p\in F}\in
^{Lim}_{\leftarrow}F$. Then 

1. If $p,q\in F$, then $x_{p\wedge q}=x_{p}\wedge x_{q}$

2. $b\in\{x_{p}|p\in F\}$ iff $b=1$ or $x_{\{b,b'\}}=b$.

3. $\{x_{p}|p\in F\}$ is an $F$-ultrafilter on $B$.
\end{lem}
\begin{proof}
1. We have $x_{p\wedge q}=a\wedge b$ for some $a\in p,b\in q$, so
$a\wedge b=x_{p\wedge q}\leq x_{p}$ and $a\wedge b\leq x_{q}$. If
$a\neq x_{p}$, then $a\wedge b=a\wedge b\wedge x_{p}=0$ a contradiction.
If $b\neq x_{q}$, then $a\wedge b=a\wedge b\wedge x_{q}=0$ a contradiction.
We therefore have $x_{p\wedge q}=a\wedge b=x_{p}\wedge x_{q}$.

2. $\leftarrow$ is trivial. For $\rightarrow$ assume $p\in F$. If
$|p|=1$, then $x_{p}=1$. If $|p|>1$, then let $q=\{x_{p},x_{p}'\}$, the
$p\preceq\{x_{p},x_{p}'\}=q$, so $x_{q}=\phi_{p,q}(x_{p})=x_{p}$.

3. Assume $p\in F$ and $x_{p}\leq a$ and $a\not\in\{x_{p}|p\in F\}$. Then $x_{\{a,a'\}}=a'$,
so $x_{p\wedge\{a,a'\}}=x_{p}\wedge x_{\{a,a'\}}=x_{p}\wedge a'=x_{p}\wedge a\wedge a'=0$
a contradiction. We therefore have $\{x_{p}|p\in F\}$ be an upper set. If
$p,q\in F$, then $x_{p}\wedge x_{q}=x_{p\wedge q}$, so $\{x_{p}|p\in F\}$ is
a filter. If $b\in B\setminus\{0,1\}$, then $x_{\{b,b'\}}=b$ or
$x_{\{b,b'\}}=b'$, so $\{x_{p}|p\in F\}$ is an ultrafilter. $\{x_{p}|p\in F\}$
is an $F$-ultrafilter since $\{x_{p}|p\in F\}\cap p$ is nonempty for each $p\in F$.
\end{proof}

Given $F$-ultrafilter $\mathcal{U}$, let $f:F\rightarrow B$
be the mapping where $f(p)$ is the unique element in $\mathcal{U}\cap p$.
If $p\preceq q$, then $\phi_{p,q}(f(p))\in q$ and
$\phi_{p,q}(f(p))\geq f(p)\in\mathcal{U}$, so $\phi_{p,q}(f(p))=f(q)$.
We therefore have $f\in ^{Lim}_{\leftarrow}\mathcal{U}$.

Define maps $L:^{Lim}_{\leftarrow}F\rightarrow S_{F}(B),
M:S_{F}(B)\rightarrow^{Lim}_{\leftarrow}F$ by letting
$L(x_{p})_{p\in F}=\{x_{p}|p\in F\}$ and where
$M(\mathcal{U})(p)\in p\cap\mathcal{U}$ for $p\in F$.

\begin{thm}
The functions $L$ and $M$ are inverses.
\end{thm}
\begin{proof}
If $(x_{p})_{p\in F}\in^{Lim}_{\leftarrow}F$, then for $p\in F$ we have
$M(L((x_{p})_{p\in F}))(p)=M(\{x_{p}|p\in F\})(p)=x_{p}$.
Let $\mathcal{U}$ be an $F$-ultrafilter. If $a\in\mathcal{U}$, then let
$p\in F$ be a partition with $a\in p$. Then $M(\mathcal{U})(p)=a$, so
$a\in L(M(\mathcal{U}))$. We therefore have $\mathcal{U}\subseteq L(M(\mathcal{U}))$, so
$\mathcal{U}=L(M(\mathcal{U}))$.
\end{proof}

A Boolean partition algebra $(B,F)$ is said to be stable if for each
$b\in B\setminus\{0\}$, there is an $(x_{p})_{p\in F}\in^{Lim}_{\leftarrow}F$
and a $p\in F$ with $b=x_{p}$.
\begin{thm}
Let $(B,F)$ be a Boolean partition algebra, then the following are equivalent.

1. $(B,F)$ is stable.

2. The projections $\pi_{p}:^{Lim}_{\leftarrow}F\rightarrow p$ are all surjective.

3. $\cup S_{F}(B)=B\setminus\{0\}$

4. $\cap S_{F}(B)=\{1\}$
\end{thm}
\begin{proof}
$3\leftrightarrow 4$. This is trivial.

$2\rightarrow 1$ Let's assume that $\pi_{p}$ is surjective. Then for each
$b\in B\setminus\{0\}$, there is a $p\in F$ with $b\in p$ and an
$(x_{p})_{p\in F}\in^{Lim}_{\leftarrow}F$ with $x_{p}=b$.

$3\rightarrow 2$ Let's assume that $p\in F$. Then for each $b\in p$ there is
a $\mathcal{U}\in S_{F}(B)$ with $b\in\mathcal{U}$, so $M(\mathcal{U})(p)=b$, so
$\pi_{p}$ is surjective.

$1\rightarrow 3$ Let's assume $(B,F)$ is stable, then for each $b\in B\setminus\{0\}$
there is an $(x_{p})_{p\in F}\in^{Lim}_{\leftarrow}F$ and a $p\in F$ with $b=x_{p}$.
We therefore have $b\in\{x_{p}|p\in F\}\in S^{*}(B,F)$.
\end{proof}

If $(P,\wedge,0),(Q,\wedge,0)$ are
semilattices and $f:P\rightarrow Q$ is a semilattice homomorphism
and $A\subseteq P\setminus\{0\}$ is a cellular family, then
for each $a,b\in A,a\neq b$ we have $f(a)\wedge f(b)=f(a\wedge b)=f(0)=0$, so
$f''(A)\setminus\{0\}$ is a cellular family. If $(B,F)$ is a Boolean partition
algebra, then write $\iota:(B,F)\rightarrow S^{*}(B,F)$ for the mapping where
$\iota(a)=\{\mathcal{U}\in S^{*}(B,F)|a\in\mathcal{U}\}$. Then $\iota$ is a Boolean
algebra homomorphism. If $p\in F$, then $\iota''(p)\setminus\{\emptyset\}$ is a partition of
$S^{*}(B,F)$ for each $p\in F$. Moreover, $\iota$ is injective iff $Ker(\iota)=0$ iff
$(B,F)$ is stable. Moreover, if $(B,F)$ is stable, then since $\iota$ is injective, for
$p,q\in F,p\neq q$ we have $\iota''(p)\neq\iota''(q)$.

\begin{thm}
If $(B,F)$ is a stable Boolean partition algebra, then for $p,q\in F$ we have
$\iota''(p\wedge q)=\iota''(p)\wedge\iota''(q)$
\end{thm}
\begin{proof}
$\iota''(p\wedge q)=
\iota''(\{a\wedge b|a\in p,b\in q\}\setminus\{0\})
=\iota''(\{a\wedge b|a\in p,b\in q\})\setminus\{\emptyset\}
=\{\iota(a\wedge b)|a\in p,b\in q\}\setminus\{\emptyset\}
=\{\iota(a)\wedge\iota(b)|a\in p,b\in q\}\setminus\{\emptyset\}
=\iota''(a)\wedge\iota''(b)$
\end{proof}

Let $(A,F)$ be a Boolean partition algebra, and let $B$ be a Boolean
algebra. Then a function $f:A\rightarrow B$ is partitional if $f$ is a
Boolean algebra homomorphism, and $f''(p)\setminus\{0\}$ is a partition of
$B$. A partition homomorphism $f:(A,F)\rightarrow(B,G)$ is a Boolean algebra
homomorphism from $A$ to $B$ where $f''(p)\setminus\{0\}\in G$ for each
$p\in F$. A function $f:(A,F)\rightarrow B$ if partitional iff $f:(A,F)\rightarrow(B,\mathbb{P}(B))$
is a partition homomorphism.
If $f:(A,F)\rightarrow B$ is an injective homomorphism, then $f$ is partitional if and
only if $f''(p)$ is a partition of $B$ for each $p\in F$. 
If $f:(A,F)\rightarrow(B,G)$ is an injective homomorphism, then $f$ is a partition homomorphism
iff $f''(p)\in G$ for each $p\in F$.

\begin{thm}
1. Let $f:(A,F)\rightarrow(B,G)$ be a partition homomorphism, and let
$g:(B,G)\rightarrow C$ be partitional. Then $g\circ f:(A,F)\rightarrow C$
is partitional as well.

2. Let $f:(A,F)\rightarrow(B,G),g:(B,G)\rightarrow(C,H)$ be partition
homomorphism, then $g\circ f$ is also a partition homomorphism.
\end{thm}
\begin{proof}
In both case $1$ and $2$, we claim that 
$(g\circ f)''(p)\setminus\{0\}=g''(f''(p)\setminus\{0\})\setminus\{0\}$.
We have $(g\circ f)''(p)\setminus\{0\}=g''(f''(p))\setminus\{0\}\supseteq
g''(f''(p)\setminus\{0\})\setminus\{0\}$. For the reverse inclusion, if
$c\in(g\circ f)''(p)\setminus\{0\}$, then $c=g(b)$ for some $b\in f''(p)$,
but since $c\neq 0$ we have $b\neq 0$ so $c=g(b)\in g''(f''(p)\setminus\{0\})\setminus\{0\}$.

1. We have $f''(p)\setminus\{0\}\in G$, so 
$(g\circ f)''(p)\setminus\{0\}=g''(f''(p)\setminus\{0\})\setminus\{0\}$ is a
partition of $C$.

2. We have $f''(p)\setminus\{0\}\in G$, so 
$(g\circ f)''(p)\setminus\{0\}=g''(f''(p)\setminus\{0\})\setminus\{0\}\in H$, so
$g\circ f$ is a partition homomorphism.
\end{proof}
It can easily be seen that if $1:(B,F)\rightarrow(B,F)$ is the identity mapping,
then $1$ a partition homomorphism. The class of all partition Boolean algebras therefore forms a category.

An extended partition is a family $(a_{i})_{i\in I}\in B^{I}$ such that $\vee_{i\in I}a_{i}=1$
and if $i\neq j$, then $a_{i}\wedge a_{j}=0$. A family $(a_{i})_{i\in I}$ is
an extended partition iff $\{a_{i}|i\in I\}$ is a partition of $B$ and
$a_{i}\neq a_{j}$ whenever $a_{i}\neq 0$.

\begin{thm}
$\label{eoaghigeoa}$

Let $A,B$ be Boolean algebras, then a function $f:A\rightarrow B$ is a
Boolean algebra homomorphism iff whenever $(a,b,c)$ is an extended partition of
$A$, then $(f(a),f(b),f(c))$ is an extended partition of $B$.
\end{thm}
\begin{proof}
$\Rightarrow$ Trivial. 

$\Leftarrow$ First take note that since
$(0,0,1)$ is an extended partition of $A$, we have $(f(0),f(0),f(1))$ be
an extended partition of $B$, so $f(0)=f(0)\wedge f(0)=0$. If
$a\in A$, then $(a,a',0)$ is an extended partition of $A$, so
$(f(a),f(a'),f(0))=(f(a),f(a'),0)$ is an extended partition of $B$, so
$f(a)'=f(a')$. 

Assume $a,b\in A$ are incompatible, then $(a,b,(a\vee b)')$ is an extended partition
of $A$, so $(f(a),f(b),f(a\vee b)')$ is an extended partition of $B$, so
$f(a)\vee f(b)=f(a\vee b)$ and $f(a)\wedge f(b)=0$.

Now assume $a\leq b$, then $f(b)=f((b\wedge a')\vee a)=f(b\wedge a')\vee f(a)$, so
$f(a)\leq f(b)$.

Therefore for arbitrary $a,b\in B$ one has $f(a)\leq f(a\vee b),f(b)\leq f(a\vee b)$, so
$f(a)\vee f(b)\leq f(a\vee b)=f((a\wedge b')\vee b)=f(a\wedge b')\vee f(b)
\leq f(a)\vee f(b)$. Therefore $f$ is a Boolean algebra homomorphism.
\end{proof}
\begin{cor}
Let $(A,F)$ be a Boolean partition algebra, and let $B$ be a Boolean
algebra. Then a function(not necessarily a Boolean algebra homomorphism)
$f:(A,F)\rightarrow B$ is partitional iff $f(0)=0$ and $(f(a))_{a\in p}$
is an extended partition of $B$.
\end{cor}
\begin{proof}
$f$ satisfies the hypothesis of theorem $\ref{eoaghigeoa}$, so $f$ is
a Boolean algebra homomorphism.
\end{proof}

\section{Uniform Spaces and Duality}
Given a set $X$, $\mathbb{P}(P(X))$ is the lattice of partitions on $X$.
We shall write $\mathbb{P}P(X)$ for $\mathbb{P}(P(X))$.
A uniform space $(X,F)$ is said to be non-Archimedean if it is generated
by equivalence relations.

A partition space is a pair $(X,M)$ where $M$ is a filter on the lattice
$\mathbb{P}P(X)$. We shall call the elements of $M$ crevasses.
Partition spaces are essentially non-Archimedean uniform spaces,
but in many circumstances partition spaces are easier to work with than
non-Archimedean uniform spaces. We shall require every complete uniform space
and complete partition space to be separating.

\begin{thm}
A separating partition space $(X,M)$ is complete iff whenever
$\phi\in ^{Lim}_{\leftarrow}M$, then there is an $x\in X$
with $x\in\phi(P)$ for each $P\in M$.
\end{thm}
\begin{proof}
$\rightarrow$ Let's assume $(X,M)$ is complete. Then let $\phi\in Lim_{\leftarrow}M$.
Then $\{\phi(R)|R\in M\}$ is an ultrafilter on  $\emptyset\cup\cup M$, so
$\{\phi(R)|R\in M\}$ is a filterbase on $X$, and $\{\phi(R)|R\in M\}$ is clearly
Cauchy. Since $(X,M)$ is complete, $\{\phi(R)|R\in M\}$ converges to some
$x\in X$, so for each $P\in M$ we have $x\in\phi(P)$.

$\leftarrow$ Let $F$ be a Cauchy filter. Then for each $P\in M$, 
there is a unique $R\in P$ with $R\in F$. Let $\phi:M\rightarrow \mathfrak{B}^{*}(X,M)$
be the function with $\phi(P)\in P,\phi(P)\in F$ for each $P\in M$. If
$P\preceq Q$, then $\phi_{P,Q}(\phi(P))\in P$, and
$\phi_{P,Q}(\phi(P))\supseteq\phi(P)\in F$, so $\phi(Q)=\phi_{P,Q}(\phi(P))$.
Therefore $\phi\in Lim_{\leftarrow}M$, so there is an $x\in X$ with $x\in\phi(P)$ for
each $P\in M$. Therefore for each neighborhood $U$ of $x$, there is a $P\in M$ and
an $R\in P$ with $x\in R\subseteq U$, but we must have $R=\phi(P)\in F$, so
$U\in F$ as well. We therefore conclude that $F\rightarrow x$.
\end{proof}

\begin{thm}
Let $(B,F)$ be a stable Boolean partition algebra. Then 

1. $\{\iota''(p)|p\in F\})$ generates a partition space structure on $S^{*}(B,F)$.

2. If $(B,F)$ is subcomplete, then $(S^{*}(B,F),\{\iota''(p)|p\in F\}))$ is
a partition space when $S^{*}(B,F)$ is nonempty.
\end{thm}
\begin{proof}
1. $\{\iota''(p)|p\in F\})$ is a filterbase since $\iota''(p)\wedge\iota''(q)=\iota''(p\wedge q).$

2. Let's assume that $(B,F)$ is subcomplete. Then assume $p\in F$ and let
$Z$ be a partition of $S^{*}(B,F)$ with $\iota''(p)\preceq Z$.

For each $R\in Z$, let $P_{R}=\{a\in p|\iota(a)\subseteq R\}$. Then let
$r=\{\vee P_{R}|R\in Z\}$. Then $r$ is a partition of $B$ refining $p$, so $r$ must
be subcomplete as well. If $\mathcal{U}\in R$, then $\mathcal{U}\in\iota(a)$ for some
$a\in p$ with $\iota(a)\subseteq R$, so $a\in\mathcal{U}$,
so $\vee P_{R}\in\mathcal{U}$, so $\mathcal{U}\in\iota(\vee P_{R})$.
We therefore have $R\subseteq\iota(\vee P_{R})$, but since
$\iota''(r)=\{\iota(\vee P_{R})|R\in Z\}$ and $Z=\{R|R\in R\}$ are both partitions,
we must have $\iota''(r)=Z\succeq p$. We therefore have $(S^{*}(B,F),\{\iota''(p)|p\in F\})$ be
a partition space.
\end{proof}

If $(B,F)$ is a Boolean partition algebra, then let
$\psi:(B,F)\rightarrow \mathfrak{B}^{*}(S^{*}(B,F))$ be the mapping given by
$\psi(x)=\iota(x)$. Given a partition space $(X,M)$, and
$x\in X$, let $\mathcal{C}(x)=\{R\in \mathfrak{B}^{*}(X,M)|x\in R\}$, then
$\mathcal{C}(x)$ is an ultrafilter on $\mathfrak{B}^{*}(X,M)$, and for each
$P\in M$ there is a unique $R\in P$ with $x\in R$, so $R\in\mathcal{C}(x)$.
We therefore have $\mathcal{C}(x)\in S^{*}(\mathfrak{B}^{*}(X,M))$ for each
$x\in X$.

\begin{thm}
1. Let $(B,F)$ be a stable Boolean partition algebra, then $S^{*}(B,F)$ is a complete
partition space.

2. If $(X,M)$ is a partition space, then $\mathfrak{B}^{*}(X,M)$ is a subcomplete and
stable Boolean partition algebra.

3. Let $(B,F)$ be a stable Boolean partition algebra, then
$\psi:(B,F)\rightarrow \mathfrak{B}^{*}(S^{*}(B,F))$ is an injective partition homomorphism, and
if $(B,F)$ also subcomplete, then $\psi$ is a partition isomorphism.

4. Let $(X,M)$ be a partition space, then $\mathcal{C}:(X,M)\rightarrow S^{*}(\mathfrak{B}^{*}(X,M))$
is uniformly continuous, and $\mathcal{C}''(X,M)$ is dense in $S^{*}(\mathfrak{B}^{*}(X,M))$.
If $(X,M)$ is separated, then $\mathcal{C}$ is a uniform embedding. If $(X,M)$ is complete,
then $\mathcal{C}$ is a uniform homeomorphism.
\end{thm}
\begin{proof}
1. 
To show that $S^{*}(B,F)$ is separated, assume $\mathcal{U},\mathcal{V}\in S^{*}(B,F)$ are distinct ultrafilters.
Then let $a\in \mathcal{U}\setminus\mathcal{V}$. Then $\{a,a'\}\in F$, but 
since $a\in\mathcal{U},a'\in\mathcal{V}$ we have $\mathcal{U}\in\iota(a),\mathcal{V}\in\iota(a')$,
so $\mathcal{U},\mathcal{V}$ are in distinct blocks of the partition
$\{\iota(a),\iota(a')\}$.

Let $M$ be the partition structure generated by $\{\iota''(p)|p\in F\}$,
and let $\varphi\in^{Lim}_{\leftarrow}M$. Then for $p\in F$ we have
$\iota''(p)\in M$, and $\varphi(\iota''(p))\in\iota''(p)$, so let $x_{p}=\iota^{-1}(\varphi(\iota''(p)))$.
Then $x_{p}\in p$ for $p\in F$. Moreover, if $p\preceq q$, then $\iota''(p)\preceq\iota''(q)$, so
$\varphi(\iota''(p))\subseteq\varphi(\iota''(q))$, so
$x_{p}=\iota^{-1}(\varphi(\iota''(p)))\leq\iota^{-1}(\varphi(\iota''(q)))=x_{q}$, so
$\varphi_{p,q}(x_{p})=x_{q}$. We therefore have $(x_{p})_{p\in F}\in^{Lim}_{\leftarrow}F$.

Let $\mathcal{U}=\{x_{p}|p\in F\}$. Then $\mathcal{U}\in S^{*}(B,F)$.
Given $P\in M$ there is a $p\in F$ with $\iota''(p)\preceq P$, so
since $x_{p}\in\mathcal{U}$ we have
$\mathcal{U}\in\iota(x_{p})=\varphi(\iota''(p))\subseteq\varphi(P)$.
We therefore have $S^{*}(B,F)$ be complete.

2. If $P\in M$ and $Z\subset P$ is non-empty,
then $P\preceq\{\cup Z,\cup(P\setminus Z)\}$, so $\cup Z\in M$, so
$\mathfrak{B}^{*}(X,M)$ is subcomplete.

To prove stability, assume $P\in M$. Then for each $R\in P$, let
$x\in R$, then let $\varphi:M\rightarrow\cup M$ be the function where
$x\in \varphi(Q)\in Q$ for each $Q\in M$. Then $\varphi\in^{Lim}_{\leftarrow}M$
and $\varphi(P)=R$. We therefore have the projection map $\pi_{P}:^{Lim}_{\leftarrow}M\rightarrow P$
be surjective. We therefore have $\mathfrak{B}^{*}(X,M)$ be stable.

3. $\psi$ is injective since $Ker\psi=Ker\iota=\{0\}$.
Let $M$ be the partition structure on $S^{*}(B,F)$. Then for $p\in F$ we have
$\psi''(p)=\iota''(p)\in M$ for each $p\in F$, so $\psi$ is a partition homomorphism.

If $(B,F)$ is subcomplete, then $M=\{\iota''(p)|p\in F\}=\{\psi''(p)|p\in F\}$ and
$B^{*}(S^{*}(B,F))=B^{*}(S^{*}(B,F),M)=(\{\emptyset\}\cup\cup M,M)$, but
$\{\emptyset\}\cup\cup M=\{\iota(b)|b\in B\}=\psi''(B)$, so $\psi$
is a partition isomorphism.

4. Take note that $\mathfrak{B}^{*}(X,M)=(\emptyset\cup\cup M,M)$, so
$S^{*}(\mathfrak{B}^{*}(X,M))=S^{*}(\emptyset\cup\cup M,M)=(S^{*}(\emptyset\cup\cup M,M),
\{\iota''(P)|P\in M\})$ since $\mathfrak{B}^{*}(X,M)$ is subcomplete. To show $\mathcal{C}$
is uniformly continuous and dense we shall take inverse images of 
the partitions $\iota''(P)$ under $\mathcal{C}$. We have
$\{\mathcal{C}_{-1}(R)|R\in\iota''(P)\}=\{\mathcal{C}_{-1}(\iota(V))|V\in P\}$.
Now $x\in\mathcal{C}_{-1}(\iota(V))$ iff $\mathcal{C}(x)\in\iota(V)$ iff
$V\in\mathcal{C}(x)$ iff $x\in V$, so $\mathcal{C}_{-1}(\iota(V))=V$, so
$\{\mathcal{C}_{-1}(R)|R\in\iota''(P)\}=\{\mathcal{C}_{-1}(\iota(V))|V\in P\}
=\{V|V\in P\}=P$. We therefore have $\mathcal{C}$ be uniformly continuous, and since
$\emptyset\not\in\{\mathcal{C}_{-1}(R)|R\in\iota''(P)\}$ for each partition
$\iota''(P)$, we have $\mathcal{C}''(X,M)\subseteq S^{*}(\mathfrak{B}^{*}(X,M))$ be dense.

If $(X,M)$ is separated, then one can clearly see that the function $\mathcal{C}$ is
injective, so since each $P\in M$ can be written as $\{\mathcal{C}_{-1}(R)|R\in\iota''(P)\}$
we have $\mathcal{C}$ be an embedding. If $(X,M)$ is complete, then
each $\mathcal{U}\in S^{*}(\mathfrak{B}^{*}(X,M))=S^{*}(\emptyset\cup\cup M,M)$, so $\mathcal{U}$
is an ultrafilter with $|\mathcal{U}\cap P|=1$ for each $P\in M$, so
$\mathcal{U}$ is Cauchy, so $\mathcal{U}\rightarrow x$ for some $x\in X$.
We therefore have $\mathcal{U}\subseteq\mathcal{C}(x)$, so $\mathcal{U}=\mathcal{C}(x)$.
We therefore have $\mathcal{C}$ be surjective, so since $\mathcal{C}$
is a uniform embedding we have $\mathcal{C}$ be a uniform homeomorphism.
\end{proof}

If $(X,M)(Y,N)$ are uniform spaces, then a mapping $f:X\rightarrow Y$ is
uniformly continuous iff $\{f_{-1}(R)|R\in Q\}\setminus\{\emptyset\}\in M$
for each $Q\in N$. If $f:X\rightarrow Y$ is uniformly continuous, then
define a mapping $\mathfrak{B}^{*}(f):\mathfrak{B}^{*}(Y,N)\rightarrow \mathfrak{B}^{*}(X,M)$ by letting
$\mathfrak{B}^{*}(f)(R)=f_{-1}(R)$. Then clearly $\mathfrak{B}^{*}(f)$ is a partition homomorphism.

\begin{thm}
If $(A,F),(B,G)$ are Boolean partition spaces, and $\phi:A\rightarrow B$
is a partition homomorphism, then for each $\mathcal{U}\in S^{*}(B,F)$
we have $\phi_{-1}(\mathcal{U})\in S^{*}(A,F)$
\end{thm}
\begin{proof}
Let's assume that $p\in F$. Then $\phi''(p)\setminus\{0\}\in G$, so
there is an $a\in p$ where $\phi(a)\in\mathcal{U}$, so $a\in\phi_{-1}(\mathcal{U})$.
\end{proof}
For each pair of Boolean partition spaces $(A,F),(B,G)$ and partition
homomorphism $\phi:(A,F)\rightarrow (B,G)$ define $S^{*}(\phi)$ by letting
$S^{*}(\phi)(\mathcal{U})=\phi_{-1}(\mathcal{U})$.

\begin{thm}
If $(A,F),(B,G)$ are stable Boolean partition algebras, and $\phi:(A,F)\rightarrow(B,G)$
is a partition homomorphism, then $S^{*}(\phi)$ is uniformly continuous.
\end{thm}
\begin{proof}
Let $P$ be a crevasse in $S^{*}(A,F)$. Then there is a $p\in F$ with $\iota''(p)\preceq P$.
We shall take the inverse image of $\iota''(p)$ under $S^{*}(\phi)$.
We have $\{S^{*}(\phi)_{-1}(R)|R\in\iota''(p)\}=\{S^{*}(\phi)_{-1}(\iota(r))|r\in p\}$.
For $\mathcal{U}\in S^{*}(B,G)$ we have $\mathcal{U}\in S^{*}(\phi)_{-1}(\iota(r))$ iff
$S^{*}(\phi)(\mathcal{U})\in\iota(r)$ iff $r\in S^{*}(\phi)(\mathcal{U})$ iff $\phi(r)\in\mathcal{U}$
iff $\mathcal{U}\in\iota(\phi(r))$, so $S^{*}(\phi)_{-1}(\iota(r))=\iota(\phi(r))$.
We therefore have $\{S^{*}(\phi)_{-1}(R)|R\in\iota''(p)\}=\{S^{*}(\phi)_{-1}(\iota(r))|r\in p\}=
\{\iota(\phi(r))|r\in p\}=\{\iota(s)|s\in\phi''(p)\}$. Therefore
$\{S^{*}(\phi)_{-1}(R)|R\in\iota''(p)\}\setminus\{\emptyset\}
=\{\iota(s)|s\in\phi''(p)\}\setminus\{\emptyset\}=
\{\iota(s)|s\in\phi''(p)\setminus\{0\}\}=\iota''(\phi''(p)\setminus\{0\})$ be a crevasse
in $S^{*}(B,G)$. We therefore have $S^{*}(\phi)$ be uniformly continuous.
\end{proof}
If $(X,L),(Y,M),(Z,N)$ are uniform spaces, and $f:X\rightarrow Y,g:Y\rightarrow Z$ are
uniformly continuous continuous, then for $R\in\mathfrak{B}^{*}(Z,M)$ we have
$\mathfrak{B}^{*}(g\circ f)(R)=(g\circ f)_{-1}(R)=f_{-1}(g_{-1}(R))=
\mathfrak{B}^{*}(f)\circ\mathfrak{B}^{*}(g)(R)$. Furthermore, if $(A,F),(B,G),(C,H)$
are Boolean partition spaces, and $f:(A,F)\rightarrow(B,G),g:(B,G)\rightarrow(C,H)$ are
partition homomorphisms, then $S^{*}(g\circ f)=(g\circ f)_{-1}(\mathcal{U})=
f_{-1}\circ g_{-1}(\mathcal{U})=S^{*}(f)\circ S^{*}(g)(\mathcal{U}).$ Therefore
$\mathfrak{B}^{*},S^{*}$ are contravariant functors since $\mathfrak{B}^{*},S^{*}$ clearly
map identity functions onto identity functions.

\begin{thm}
1. Let $(X,M),(Y,N)$ be uniform spaces, and let $f:X\rightarrow Y$ be uniformly continuous,
then $S^{*}(\mathfrak{B}^{*}(f))\circ \mathcal{C}_{(X,M)}=C_{(Y,N)}\circ f$.

2. Let $(A,F),(B,G)$ be a stable Boolean  algebras, and let $f:A\rightarrow B$ be a
partition homomorphism, then $\mathfrak{B}^{*}(S^{*}(f))\psi_{(A,F)}=\psi_{(B,G)}f$.

3. If $(X,M)$ is a uniform space, then the functions
$\mathfrak{B}^{*}(\mathcal{C}_{X}):\mathfrak{B}^{*}(S^{*}(\mathfrak{B}^{*}(X,M)))\rightarrow \mathfrak{B}^{*}(X,M)$ and
$\psi:\mathfrak{B}^{*}(X,M)\rightarrow \mathfrak{B}^{*}(S^{*}(\mathfrak{B}^{*}(X,M)))$ are inverses.

4. If $(B,F)$ is a stable Boolean partition algebra, then
$S^{*}(\psi_{B}):S^{*}(\mathfrak{B}^{*}(S^{*}(B,F)))\rightarrow S^{*}(B,F)$ and
$\mathcal{C}:S^{*}(B,F)\rightarrow S^{*}(\mathfrak{B}^{*}(S^{*}(B,F)))$ are inverses.
\end{thm}
\begin{proof}
1. For $x\in X$ we have
$S^{*}(\mathfrak{B}^{*}(f))C_{(X,M)}(x)=\mathfrak{B}^{*}(f)_{-1}(\{R\in B^{*}(X,M)|x\in R\})
=\{S\in B^{*}(Y,N)|x\in\mathfrak{B}^{*}(f)(S)\}=\{S\in\mathfrak{B}^{*}(Y,N)|x\in f_{-1}(S)\}
=\{S\in\mathfrak{B}^{*}(Y,N)|f(x)\in S\}=\mathcal{C}_{(Y,M)}\circ f(x)$.

2. $\mathfrak{B}^{*}(S^{*}(f))\psi_{A}(a)=
\mathfrak{B}^{*}(S^{*}(f))(\{\mathcal{U}\in S^{*}(A,F)|a\in\mathcal{U}\})=
(S^{*}(f)_{-1}(\{\mathcal{U}\in S^{*}(A,F)|a\in\mathcal{U}\})
=\{\mathcal{V}\in S^{*}(B,G)|a\in S^{*}(f)(\mathcal{V})\}=
\{\mathcal{V}\in S^{*}(B,G)|a\in f_{-1}(\mathcal{V})\}=
\{\mathcal{V}\in S^{*}(B,G)|f(a)\in\mathcal{V}\}=\psi_{B}\circ f(a).$

3. We shall show that $\mathcal{B}^{*}(\mathcal{C})\psi:\mathfrak{B}^{*}(X,M)\rightarrow
\mathcal{B}^{*}(X,M)$ is the identity function. If $R\in\mathfrak{B}^{*}(X,M)$,
then $x\in\mathfrak{B}^{*}(\mathcal{C})\psi(R)=\mathcal{C}_{-1}\psi(R)$ iff
$\mathcal{C}(x)\in\psi(R)$ iff $R\in\mathcal{C}(x)$ iff $x\in R$, so
$\mathfrak{B}^{*}(\mathcal{C})\psi(R)=R$, so $\mathfrak{B}^{*}(\mathcal{C})$ is
the identity function.

4. We shall show that $S^{*}(\psi)\mathcal{C}:S^{*}(B,F)\rightarrow S^{*}(B,F)$ is
the identity function. Let $\mathcal{U}\in S^{*}(B,F)$. Then $a\in S^{*}(\psi)\mathcal{C}(\mathcal{U})$
iff $\psi(a)\in\mathcal{C}(\mathcal{U})$ iff $\mathcal{U}\in\psi(a)$ iff $a\in\mathcal{U}$.
We therefore have $S^{*}(\psi)\mathcal{C}$ be the identity function.
\end{proof}
It should be noted that every compact space has a unique uniform structure.
More specifically, if $X$ is compact, then let $F$ be the filter on $X\times X$
where $R\in F$ if $R$ is a neighborhood of the diagonal. Then $F$
is the unique uniformity on $X$. Furthermore, $(X,\mathcal{U})$ is a uniform space,
then $X$ is compact iff $(X,\mathcal{U})$ is complete and totally bounded.
The duality between compact totally disconnected spaces and Boolean algebras
follows as a consequence of these facts.

\bibliographystyle{amsplain}
\bibliography{}
\end{document}